\documentclass[11pt]{amsart}

\usepackage{amsfonts, amsmath, amscd}
\usepackage[psamsfonts]{amssymb}

\usepackage{amssymb}

\usepackage{pb-diagram}

\usepackage[all,cmtip]{xy}

\usepackage[usenames]{color}

\newtheorem{theorem}{Theorem}[section]

\newtheorem{corollary}[theorem]{Corollary}
\newtheorem{question}[theorem]{Question}
\newtheorem{remark}[theorem]{Remark}
\newtheorem{proposition}[theorem]{Proposition}
\newtheorem{definition}[theorem]{Definition}
\newtheorem{example}[theorem]{Example}

\numberwithin{equation}{section}

\newcommand{\CC}{C_k}
\newcommand{\NN}{\mathbb{N}}

\newcommand{\GG}{\mathfrak{G}}

\newcommand{\UU}{\mathcal{U}}

\newcommand{\w}{\omega}

\newcommand{\TTT}{\mathcal{T}}

\newcommand{\KK}{\mathcal{K}}

\newcommand{\Nn}{\mathcal{N}}

\newcommand{\AAA}{\mathcal A}
\newcommand{\IR}{\mathbb{R}}

\newcommand{\e}{\varepsilon}

\renewcommand{\phi}{\varphi}

\newcommand{\U}{\mathcal U}

\input xy
\xyoption{all}

\title[Fundamental bounded resolutions and quasi-$(DF)$-spaces]{Fundamental bounded resolutions and quasi-$(DF)$-spaces}

\author{J. C. Ferrando}
\address{Centro de Investigaci\'{o}n Operativa, Edificio Torretamarit, Avda
de la Universidad, Elche, Spain}
\email{jc.ferrando@umh.es}

\author{S. Gabriyelyan}
\address{Department of Mathematics, Ben-Gurion University of the
Negev, Beer-Sheva, P.O. 653, Israel}
\email{saak@math.bgu.ac.il}

\author{J. K\c akol}
\address{ Faculty of Mathematics and Informatics. A. Mickiewicz University,
61-614 Pozna\'n, and Institute of Mathematics Czech Academy of Sciences, Prague}
\email{kakol@amu.edu.pl}
\thanks{The third  named author supported by GA\v{C}R Project 16-34860L and RVO: 67985840.}


\begin{document}

\begin{abstract}

We introduce a new class of locally convex spaces $E$, under the name quasi-$(DF)$-spaces,  containing strictly the class of $(DF)$-spaces. A locally convex space $E$ is called a quasi-$(DF)$-space if (i) $E$ admits a fundamental bounded resolution, i.e. an $\mathbb{N}^{\mathbb{N}}$-increasing  family of bounded sets in $E$ which swallows all bounded set in $E$, and (ii) $E$ belongs to the class $\GG$ (in sense of Cascales--Orihuela). The class of quasi-$(DF)$-spaces is closed under taking subspaces, countable direct sums and countable products. Every regular  $(LM)$-space (particularly, every metrizable locally convex space) and its strong dual are quasi-$(DF)$-spaces. We prove  that $C_{p}(X)$ has a fundamental bounded resolution iff $C_{p}(X)$ is a quasi-$(DF)$-space  iff the strong dual of $C_{p}(X)$ is a quasi-$(DF)$-space iff $X$ is countable.  If $X$ is a metrizable space, then $\CC(X)$ is a  quasi-$(DF)$-space  iff $X$ is a Polish $\sigma$-compact space. We provide  numerous  concrete examples which in particular  clarify differences between $(DF)$-spaces and quasi-$(DF)$-spaces.
\end{abstract}

\maketitle

%%%%%%%%%%%%%%%%%%%%%%%%%%%
%%%%%%%%%%%%%%%%%%%%%%%%%%%
%%%%%%%%%%%%%%%%%%%%%%%%%%%
%%%%%%%%%%%%%%%%%%%%%%%%%%%

\section{Introduction}

%%%%%%%%%%%%%%%%%%%%%%%%%%%
%%%%%%%%%%%%%%%%%%%%%%%%%%%
%%%%%%%%%%%%%%%%%%%%%%%%%%%
%%%%%%%%%%%%%%%%%%%%%%%%%%%

An important class of locally convex spaces (lcs, for short), the class of $(DF)$-space,  was introduced by Grothendieck in \cite{Grot-54}, we refer also to  monographs  \cite{Jar} or \cite{bonet} for more details.
\begin{definition} \label{def:DF-space} {\em
A locally convex space $E$ is called a \emph{$(DF)$-space} if
\begin{enumerate}
\item[(1)] $E$ has a fundamental  sequence of bounded sets, and
\item[(2)] $E$ is $\aleph_{0}$-quasibarrelled. % i.e. every  bornivorous closed absolutely convex subset of $E$ which can be  represented as the intersection of a sequence of closed absolutely convex neighbourhoods of zero is itself a neighbourhoods of zero, see \cite{bonet}.
\end{enumerate} }
\end{definition}
All countable inductive limits of normed spaces (hence all normed spaces)  are $(DF)$-spaces. Less evident facts are  that the strong dual of a metrizable lcs is a $(DF)$-space and the strong dual of a $(DF)$-space is a metrizable and complete lcs, see \cite[Theorem 8.3.9,  Proposition 8.3.7]{bonet}.

The concept of $(DF)$-spaces has been generalized by Ruess \cite{ru1,ru2} (under the name $(gDF)$-spaces) and has also intensively studied by Noureddine \cite{Nou1,Nou2} who called them $D_{b}$-spaces. Their papers provided also some information about spaces which in \cite{Jar} were called $(df)$-spaces. Next Adasch and Ernst \cite{adasch1,adasch2}  provided another line of research around  $(DF)$-spaces in the setting of general topological vector spaces. Note however that  all known generalizations of $(DF)$-spaces $E$ kept up the  condition (1) on $E$ to have a fundamental sequence of bounded sets while some variations of the  weak barrelledness condition (2) on $E$ have been assumed.

In the present paper we propose another very natural generalization of the concept of $(DF)$-spaces. We replace the quite strong and demanding condition on $E$ to have a fundamental sequence of bounded sets by a weaker one called a \emph{fundamental bounded resolution} and assume some natural extra property on the weak*-dual of $E$ (which holds for $(DF)$-spaces), see Definition \ref{def:quasi-DF}.

Let $\{ B_{n}\}_{n\in\NN}$ be a fundamental sequence of bounded subsets of a lcs $E$. For each $\alpha=(n_{k})\in\mathbb{N}^{\mathbb{N}}$ set $B_{\alpha}:=B_{n_{1}}$. Then  the family $\mathcal{B}=\{B_{\alpha}:\alpha\in\mathbb{N}^{\mathbb{N}}\}$  satisfies the conditions:
\begin{enumerate}
\item [(i)]  every set $B_{\alpha}$ is bounded in $E$;
\item[(ii)]  $\mathcal{B}$ is a resolution in $E$ (i.e., $\mathcal{B}$ covers $E$ and $B_{\alpha}\subseteq B_{\beta}$ if $\alpha\leq\beta$, $\alpha,\beta\in\mathbb{N}^{\mathbb{N}}$);
\item[(iii)]  every bounded subset of  $E$ is contained in  some $B_{\alpha}$.
\end{enumerate}
If a lcs $E$ is covered by  a family $\{B_{\alpha}:\alpha\in\mathbb{N}^{\mathbb{N}}\}$ of bounded sets satisfying conditions (i) and (ii) [and (iii)] we shall say that $E$ has a \emph{[fundamental] bounded resolution}, see also \cite{KM} for more details. Let us recall that every metrizable lcs $E$ has a fundamental bounded resolution (see, for example, Corollary \ref{c:metrizable-fbr} below), while $E$ has a fundamental sequence of bounded sets if and only if $E$ is normable.

The definition of a quasi-$(DF)$-space involves also the following concept due to Cascales and Orihuela, see \cite{CasOr}.
\begin{definition}[Cascales--Orihuela] {\em
A lcs $E$ belongs to {\em the class $\mathfrak{G}$} if there is a resolution $\{A_{\alpha }:\alpha \in \mathbb{N}^{\mathbb{N}}\}$ in the weak*-dual $(E^{\prime },\sigma (E^{\prime },E))$ of $E $ such that each sequence in any $A_{\alpha }$ is equicontinuous.}
\end{definition}
Particularly, every set $A_{\alpha }$ is relatively $\sigma (E^{\prime},E)$-countably compact. The class $\mathfrak{G}$ is indeed large and contains `almost all' important locally convex spaces (including $(DF)$-spaces  and even dual metric spaces). Furthermore the class $\mathfrak{G}$ is stable under taking
subspaces, completions, Hausdorff quotients and countable direct sums and products, see \cite{CasOr} or \cite{kak}.

In \cite{CKS} the authors introduced and studied a subclass of lcs in the class $\GG$:
\begin{definition}[Cascales--K\c akol--Saxon] {\em
A lcs $E$ is said to have a {\em $\GG$-base} if $E$ admits a base $\{U_{\alpha}:\alpha\in\mathbb{N}^{\mathbb{N}}\}$ of neighbourhoods of zero such that $U_{\alpha}\subseteq U_{\beta}$ for all $\beta\leq\alpha$.}
\end{definition}
They proved in \cite[Lemma 2]{CKS} that a quasibarrelled lcs $E$  has a $\GG$-base if and only if $E$ is in class $\GG$.  However, under $\aleph_{1}<\mathfrak{b}$ there is a $(DF)$-space (hence belongs to the class $\GG$) which does not admit a $\GG$-base, see \cite{kak} (or Example \ref{exa:DF-no-G-base} below). Note also that every  $(LM)$-space $E$ has a $\GG$-base, so $E$ is in the class $\GG$.

Now we define quasi-$(DF)$-spaces.

\begin{definition}\label{def:quasi-DF} {\em
A lcs $E$ is called a {\em quasi-$(DF)$-space} if
\begin{enumerate}
\item [(i)] $E$ admits a fundamental bounded resolution;
\item [(ii)] $E$ belongs to the class $\GG$.
\end{enumerate} }
\end{definition}
Note that (i) and (ii) are independent in the sense that there exist lcs $E$ in the class $\GG$ without a fundamental bounded resolution, and there exist lcs  $E$ not being in the class $\GG$ but having a fundamental bounded resolution, see Examples \ref{exa:def-DF-space-1} and \ref{exa:def-DF-space-2} below.

%The above definition seems to be optimal by the following reasons.

The class of quasi-$(DF)$-spaces is stable under taking subspaces, countable direct sums and countable products (see Theorem \ref{t:properties-quasi-DF-spaces}), while infinite products of $(DF)$-spaces are not of that type.

The organization of the paper goes as follows. Section \ref{sec:general} provides  a dual characterization of lcs  with fundamental bounded resolutions, see Theorem \ref{t:G-base-strong}. This result nicely applies to show that every  regular $(LM)$-space (for the definition see below) has a  fundamental bounded resolution.  % Theorem  \ref{t:LCS-bounded-res} shows  that a lcs $E$ has a bounded resolution if and only if  $E$ with the weak topology  $w:=\sigma(E,E')$ (shortly $E_{w}$) is contained in a $K$-analytic subspace of  $\IR^\kappa$ for some cardinal $\kappa$.
We  provide many concrete examples in order to highlight the differences among some properties or between the concepts which have been introduced, for example concrete examples of  lcs having a bounded resolution but not admitting a  fundamental bounded resolution will be examined.

Denote by $C_{p}(X)$ and $\CC(X)$ the space $C(X)$ of continuous real-valued functions on a Tychonoff space $X$ endowed with the pointwise topology and the compact-open topology, respectively. The main result of Section \ref{sec:function} states that $C_p(X)$ has  a fundamental bounded resolution if and only if $X$ is countable if and only if the strong dual of $C_p(X)$ has a fundamental bounded resolution, see Theorem \ref{t:FBR-Cp} and Proposition \ref{p:strong-dual-Cp-fund-res}. This result shows that, although the class $\GG$ of lcs is relatively large, the existence of a fundamental bounded resolution is  rather  a strong restricted condition  for non-metrizable lcs.

Last section gathers  some properties of quasi-$(DF)$-spaces (see Theorem \ref{t:properties-quasi-DF-spaces}) and extends a result of Corson in \cite{Mich}. The previous sections  apply  to conclude  that $C_{p}(X)$ is a quasi-$(DF)$-space if and only if $C_{p}(X)$ has a fundamental bounded resolution if and only if $X$ is countable if and only if the strong dual of $C_{p}(X)$ is a quasi-$(DF)$-space. The latter result   combined with   Theorem \ref{t:FBR-Cp} shows that, for spaces $C_{p}(X)$, both conditions (i) and (ii) from Definition \ref{def:quasi-DF} fail for uncountable $X$.  Recall here  that $C_{p}(X)$ is a $(DF)$-space only if $X$ is finite, see \cite{schmets}. On the other hand, if $X$ is a metrizable space, then    $\CC(X)$ is a quasi-$(DF)$-space if and only if $X$ is a Polish $\sigma$-compact space, see Proposition \ref{p:Ck-quasi-DF-metr}. This shows also an essential difference between quasi-$(DF)$-spaces and $(DF)$-spaces $C_{k}(X)$.

Recall that the strong dual $E'_{\beta}:=(E',\beta(E',E))$ of a metrizable lcs $E$ is a $(DF)$-space. However, this result fails even for  strict $(LF)$-spaces $E$ since  $E'_{\beta}$  need not be $\aleph_{0}$-quasibarrelled, see \cite[Proposition~1]{bonet2}. %, so the necessary condition for $E'_{\beta}$  to be a $(DF)$-space fails for this case.
Since for a non-metrizable  strict $(LF)$-space   $E$  the space $E'_{\beta}$ does not have a fundamental sequence of bounded sets,  the strong dual $E'_{\beta}$ is not a $(DF)$-space but it is  a quasi-$(DF)$-space, see Theorem \ref{t:properties-quasi-DF-spaces}.

%%%%%%%%%%%%%%%%%%%%%%%%%%%
%%%%%%%%%%%%%%%%%%%%%%%%%%%
%%%%%%%%%%%%%%%%%%%%%%%%%%%
%%%%%%%%%%%%%%%%%%%%%%%%%%%

\section{[Fundamental] bounded resolutions and  $\GG$-bases; general case} \label{sec:general}

%%%%%%%%%%%%%%%%%%%%%%%%%%%
%%%%%%%%%%%%%%%%%%%%%%%%%%%
%%%%%%%%%%%%%%%%%%%%%%%%%%%
%%%%%%%%%%%%%%%%%%%%%%%%%%%

In this section we obtain dual characterizations of the existence of a [fundamental] bounded resolutions in a lcs $E$. % This sections examine the relationship between existence of [fundamental] bounded resolutions  for a lcs $E$ and $\GG$-bases for the strong dual $E'_{\beta}$ of $E$.

Let $I$ be a partially ordered set. A family $\AAA=\{ A_i\}_{i\in I}$ of subsets of a set $\Omega$ is called {\em $I$-increasing} ({\em $I$-decreasing}) if $A_i \subseteq A_j$ ($A_i \supseteq A_j$, respectively) for every $i\leq j$ in $I$. We say that the family $\AAA$ {\em swallows } a family $\mathcal{B}$ of subsets of $\Omega$ if for every $B\in \mathcal{B}$ there is an $i\in I$ such that $B\subseteq A_i$. An $\NN^\NN$-increasing family of subsets of $\Omega$ is called a {\em resolution} in $\Omega$ if it covers $\Omega$.
A resolution in a topological space $X$ is called {\em compact} (respectively, {\em fundamental compact}) if all its elements are compact subsets of $X$ (and it swallows compact subsets of $X$, respectively).

We start with the following general observation.

\begin{proposition} \label{p:LCS-bounded-resol}
If a lcs  $E$  admits  an $I$-decreasing base at zero, then  $E$ has an $\NN^I$-increasing bounded resolution swallowing bounded subsets of $E$. Consequently $E'_\beta$ has an $\NN^I$-decreasing base at zero.
\end{proposition}

\begin{proof}
Let $\UU:=\{ U_i: i\in I\}$ be an $I$-decreasing base at zero in $E$. For every $\alpha\in  \NN^I$, set
\[
B_\alpha := \bigcap_{i\in I} \alpha(i)U_i,
\]
and set $\mathcal{B}:=\{ B_\alpha: \alpha\in \NN^I\}$. Clearly, $\mathcal{B}$ is $\NN^I$-increasing bounded resolution in $E$. To show that $\mathcal{B}$ swallows the bounded sets of $E$, fix a bounded subset $B$ of $E$. For every $i\in I$, choose a natural number $\alpha(i)$ such that $B\subseteq \alpha(i)U_i$ and set $\alpha:=\big(\alpha(i)\big)\in \NN^I$. Clearly, $B\subseteq B_\alpha$.
\end{proof}

Since a metrizable lcs has an $\NN$-decreasing base, Proposition \ref{p:LCS-bounded-resol} implies
\begin{corollary} \label{c:metrizable-fbr}
If $E$ is a metrizable lcs, then $E$ has a fundamental bounded resolution and hence the space $E'_\beta$ has a $\GG$-base.
\end{corollary}

\begin{example} \label{exa:BR-R-kappa} {\em
For every uncountable cardinal $\kappa$, the space $\IR^\kappa$ does not have bounded resolution. Indeed, assuming the converse we obtain that the complete space $\IR^\kappa$ has a fundamental bounded resolution by Valdivia's theorem \cite[Theorem~3.5]{kak}. But then  the closures of the sets of this latter family compose a fundamental compact resolution. Now Tkachuk's theorem \cite[Theorem~9.14]{kak} implies that $\kappa$ is countable, a contradiction.}
\end{example}

Next proposition  gathers  the most important stability properties of spaces with a fundamental bounded resolution.
\begin{proposition} \label{p:Stability-FBR}
The class of locally  convex spaces with a fundamental bounded resolution is closed under taking (i) subspaces, (ii) countable direct sums, and (iii) countable products.
\end{proposition}

\begin{proof}
(i) is clear. To prove (ii) and (iii) we shall use the following encoding operation of elements of $\mathbb{N}^{\mathbb{N}}$. We encode each $\alpha \in \mathbb{N}^{\mathbb{N}}$ into a sequence $\{\alpha _{i}\}_{i\in \mathbb{N}}$ of elements of $\mathbb{N}^{\mathbb{N}}$ as follows. Consider an arbitrary decomposition of $\mathbb{N}$ onto a disjoint family  $\{N_{i}\}_{i\in \mathbb{N}}$ of infinite sets, where $N_{i}=\{n_{k,i}\}_{k\in \mathbb{N}}$ for $i\in \mathbb{N}$. Now for  $\alpha =\left( \alpha (n)\right) _{n\in \mathbb{N}}$ and $i\in \mathbb{N}$, we set $\alpha_{i}=\left( \alpha_{i}(k)\right) _{k\in \mathbb{N}}$, where $\alpha _{i}(k):=\alpha (n_{k,i})$ for every $k\in \mathbb{N}$. Conversely, for every sequence  $\{\alpha_{i}\}_{i\in \mathbb{N}}$ of  elements of $\mathbb{N}^{\mathbb{N}}$, we define $\alpha =\left( \alpha (n)\right)_{n\in \mathbb{N}}$ setting $\alpha(n):=\alpha _{i}(k)$ if $n=n_{k,i}$.

(ii) Let $E=\prod_{i\in\NN} E_i$, where every $E_i$ has a fundamental bounded resolution $\{ B^i_\alpha:\alpha\in\NN^\NN\}$. For every $\alpha\in\NN$, we define
$B_\alpha :=\prod_{i\in\NN} B^i_{ \alpha_i}$
and  set $\mathcal{B}:=\{ B_\alpha:\alpha\in\NN^\NN\}$.  Since a subset of $E$ is bounded if and only if  its projection onto $E_i$ is bounded in $E_i$ for every $i\in\NN$, it is easy to see that $\mathcal{B}$ is a fundamental bounded resolution in $E$.

(iii) Let $E=\bigoplus_{i\in\NN} E_i$, where every $E_i$ has a fundamental bounded resolution $\{ B^i_\alpha:\alpha\in\NN^\NN\}$. For every $\alpha\in\NN^\NN$, set $\alpha^\ast := \big(\alpha(k+1)\big)_{k\in\NN}$ and
$B_\alpha := \prod_{i=1}^{\alpha(1)} B^i_{\alpha^\ast_i}.$
Since every bounded subset of $E$ is contained and bounded in $\bigoplus_{i=1}^m E_i$ for some $m\in\NN$ (see  Proposition 24 of Chapter 5 of \cite{RobRob}), we obtain that the family $\{ B_\alpha:\alpha\in\NN^\NN\}$ is  a fundamental bounded resolution in $E$.
\end{proof}

\begin{example} {\em
Let $E$ be an infinite-dimensional Banach space and let $B$ be the closed unit ball of $E$. Then its dimension is uncountable. Choose a Hamel basis $H = \{ x^\ast_{i} : i \in I\}$ from its topological dual $E'$. Then the map $x\to \big(x^\ast_{i}(x)\big)$ from $E_w$ into $\IR^H$ is an embedding with dense image. Since $E$ is a Banach space, the sequence $\{ nB\}_{n\in\NN}$ is a fundamental bounded sequence in $E_w$. However, since $H$ is uncountable, the completion $\IR^H$ of $E_w$ does not have even a bounded resolution by Example \ref{exa:BR-R-kappa}. Hence, the completion of a lcs with a fundamental bounded resolution in general does not have a bounded resolution. }
\end{example}

%\begin{example} {\em
%Let $H$ be a Banach space containing $\ell_{1}(\IR)$. Then consider $E:= (H', \sigma(H',H))$. Since $H$ contains $\ell_{1}(\IR)$, the dual unit ball $B^\ast$  is not sequentially compact in $E$, see page 226 of \cite{Diestel}. Since $B^\ast$ is compact in $E$, the space $E$ has  a bounded (even compact) resolution $\{ nB^\ast\}_{n\in\NN}$. However, $E$ is not angelic because the compact set $B^\ast$ is not sequentially compact. }
%\end{example}

For a subset $A$ of a lcs $E$, we denote by $A^\circ$ the polar of $A$ in $E'$. Below we give a dual characterization of lcs with a fundamental bounded resolution.
Recall that a lcs $E$ is a \emph{quasi-$(LB)$-space } \cite{Va1} if $E$ admits a resolution consisting of Banach discs of $E$.
\begin{theorem} \label{t:G-base-strong}
For a lcs $E$ the following assertions are equivalent:
\begin{enumerate}
\item[{\rm (i)}] $E$ has a fundamental bounded resolution;
\item[{\rm (ii)}] the strong dual $E'_\beta$ of $E$ has a $\mathfrak{G}$-base;
\item[{\rm (iii)}] the weak* bidual $\left( E^{\prime \prime }, \sigma \left( E'',E^{\prime }\right) \right)$ is a quasi-$\left( LB\right) $-space.
\end{enumerate}
If in addition the space $E$ is locally complete, then (i)-(iii) are equivalent to
\begin{enumerate}
\item[{\rm (iv)}] $E$ has a bounded resolution;
\item[{\rm (v)}] $E$ is a quasi-$(LB)$-space.
\end{enumerate}
\end{theorem}

\begin{proof}
(i)$\Rightarrow$(ii) Let $\left\{ A_{\alpha }:\alpha \in \mathbb{N}^{\mathbb{N}}\right\} $ be  a fundamental  bounded resolution. Then, as easily seen,  the family $\left\{ A_{\alpha }^\circ:\alpha \in \mathbb{N}^{\mathbb{N}}\right\} $ is a $\mathfrak{G}$-base of neighborhoods of the strong topology on $E^{\prime }$.

(ii)$\Rightarrow$(i) If $\left\{ V_{\alpha }:\alpha \in \mathbb{N}^{\mathbb{N}}\right\} $ is a $\mathfrak{G}$-base of neighborhoods of $\beta \left( E^{\prime },E\right) $, the polars $\left\{ V_{\alpha }^{\circ}:\alpha \in \mathbb{N}^{\mathbb{N}}\right\} $ of the sets $V_{\alpha }$ in the bidual $E^{\prime \prime }$ of $E$ compose a compact resolution on $E^{\prime \prime }$ for the weak* topology. Setting $B_{\alpha }:=V_{\alpha }^{\circ}\cap E$ for every $\alpha \in\mathbb{N}^{\mathbb{N}}$, we can see that the family $\left\{ B_{\alpha}:\alpha \in \mathbb{N}^{\mathbb{N}}\right\} $ is a bounded resolution for $\left( E,\sigma \left( E,E^{\prime }\right) \right) $, hence for $E$. If $Q$ is a closed absolutely convex bounded subset of $E$, the polar $Q^{\circ}$ of $Q$ in $E^{\prime }$ is a neighborhood of the origin in $\beta  \left( E^{\prime},E\right) $ and consequently there exists $\beta \in \mathbb{N}^{\mathbb{N}} $ such that $V_{\beta }\subseteq Q^{\circ}$. This implies that $Q^{\circ\circ}\subseteq V_{\beta }^{\circ}$, the polars being taken in $E^{\prime \prime }$. Hence
\begin{equation*}
Q=\overline{Q}^{\sigma \left( E,E^{\prime }\right) }=Q^{\circ\circ}\cap E\subseteq B_{\beta },
\end{equation*}%
which means that the family $\left\{ B_{\alpha }:\alpha \in \mathbb{N}^{\mathbb{N}}\right\} $ swallows the bounded sets of $E$.

(i)$\Rightarrow$(iii) If $\left\{ A_{\alpha }:\alpha \in \mathbb{N}^{\mathbb{N}}\right\} $ is a fundamental bounded resolution of $E$, it is clear that
\begin{equation*}
E^{\prime \prime }=\bigcup \left\{ A_{\alpha }^{\circ\circ}:\alpha \in \mathbb{N}^{\mathbb{N}}\right\} .
\end{equation*}%
Since each $A_{\alpha }^{\circ\circ}$ is absolutely convex and weak* compact, it is a Banach disc. So $\left( E^{\prime \prime }, \sigma \left( E'',E^{\prime }\right) \right) $ has a resolution consisting of weak* compact Banach discs, which means that the weak* bidual $\left( E^{\prime \prime }, \sigma \left( E'',E^{\prime }\right) \right)$ is a quasi-$\left( LB\right) $-space.

(iii)$\Rightarrow$(i) If $\left( E^{\prime \prime }, \sigma \left( E'',E^{\prime }\right) \right) $ is a quasi-$\left( LB\right) $-space, then, by \cite[Theorem~3.5]{kak}, there is a resolution $\left\{ D_{\alpha }:\alpha \in \mathbb{N}^{\mathbb{N}}\right\} $ for $\left( E^{\prime \prime }, \sigma \left( E'',E^{\prime }\right) \right)$ consisting of Banach discs that swallows the Banach discs of $\left( E^{\prime \prime }, \sigma \left( E'',E^{\prime }\right) \right)$. If $Q$ is a bounded subset of $E$, then $Q^{\circ\circ}$ is a weak* compact Banach disc in $\left( E^{\prime \prime }, \sigma \left( E'',E^{\prime }\right) \right)$. Hence, there is $\gamma \in \mathbb{N}^{\mathbb{N}}$ such that $Q^{\circ\circ}\subseteq D_{\gamma }$, so that $Q\subseteq D_{\gamma }\cap E$. Setting $B_{\alpha }:=D_{\alpha }\cap E$ for each $\alpha \in \mathbb{N}^{\mathbb{N}}$, the family $\left\{ B_{\alpha }:\alpha \in \mathbb{N}^{\mathbb{N}}\right\} $ is a fundamental bounded resolution for $E$.

Assume that $E$ is locally complete. Clearly, (i) implies (iv). Let us show that (iv) implies (v). Since $E$ is locally complete, the absolutely convex closed envelope of any bounded set of $E$ is a Banach disc by Proposition 5.1.6 of \cite{bonet}. Thus the space $E$ is a quasi-$(LB)$-space.
%The implication (iv)$\Rightarrow$(ii) follows from Proposition 5.10 of \cite{GKL}.

(v)$\Rightarrow$(i)  By Valdivia's   theorem \cite[Theorem~3.5]{kak}, there exists another quasi-$(LB)$-representation $\mathcal{B}=\{ B_\alpha:\alpha\in\NN^\NN\}$ of $E$ swallowing all Banach discs of $E$. Since each bounded set of $E$ is contained in a Banach disc and each Banach disc is bounded, we obtain that $\mathcal{B}$ is a bounded resolution which swallows all bounded sets of $E$.
\end{proof}

Below we apply Theorem \ref{t:G-base-strong} to function spaces $\CC(X)$ and $C_p(X)$.
%\begin{corollary} \label{c:Ck-strong-dual-G}
%Let $X$ be such that $\CC(X)$ is locally complete (for instance, $X$ is a $k_\IR$-space). Then the following assertions are equivalent:
%\begin{enumerate}
%\item[{\rm (i)}] $\CC(X)$ has  a bounded  resolution;
%\item[{\rm (ii)}] $\CC(X)$ has  a  fundamental bounded  resolution;
%\item[{\rm (iii)}] the strong dual of $\CC(X)$ has a $\GG$-base;
%\item[{\rm (iv)}] $\CC(X)$ is a quasi-$(LB)$-space.
%\end{enumerate}
%If in addition $X$ is metrizable, then (i)-(iv) are equivalent to
%\begin{enumerate}
%\item[{\rm (v)}] $C_p(X)$ has  a bounded  resolution;
%\item[{\rm (vi)}] $X$ is $\sigma$-compact.
%\end{enumerate}
%\end{corollary}

%\begin{proof}
%The equivalences (i)-(iv) follow from Theorem \ref{t:G-base-strong}. The equivalence (v)$\Leftrightarrow$(vi) is Corollary 9.2 of \cite{kak}, and the implication (i)$\Rightarrow$(v) is trivial. The implication (vi)$\Rightarrow$(i) follows from Corollary 2.10 of \cite{GK-Free-LCS-G} which states that $\CC(X)$ has even a fundamental compact resolution.
%\end{proof}

\begin{corollary} \label{c:Ck-strong-dual-G}
Let $X$ be such that $\CC(X)$ is locally complete (for instance, $X$ is a $k_\IR$-space). Then the following assertions are equivalent:
\begin{enumerate}
\item[{\rm (i)}] $\CC(X)$ has  a bounded  resolution;
\item[{\rm (ii)}] $\CC(X)$ has  a  fundamental bounded  resolution.
\end{enumerate}
If in addition $X$ is metrizable, then (i)-(iv) are equivalent to
\begin{enumerate}
%\item[{\rm (v)}] $\CC(X)$ has  a compact  resolution swallowing the compact sets;
%\item[{\rm (vi)}] $\CC\big(\CC(X)\big)$ has a $\GG$-base;
\item[{\rm (iii)}] $C_p(X)$ has  a bounded  resolution;
\item[{\rm (iv)}] $X$ is $\sigma$-compact.
\end{enumerate}
\end{corollary}

\begin{proof}
The equivalences (i)-(ii) follow from Theorem \ref{t:G-base-strong}. The equivalence (iii)$\Leftrightarrow$(iv) is Corollary 9.2 of \cite{kak}, and the implication (i)$\Rightarrow$(iii) is trivial. The implication (iv)$\Rightarrow$(i) follows from Corollary 2.10 of \cite{GK-Free-LCS-G} which states that $\CC(X)$ has even a fundamental compact resolution.
\end{proof}

Observe that if $\CC(X)$  has a fundamental bounded resolution,  $\CC(X)$ need not be metrizable. For instance, $\CC(Q)$ has a fundamental bounded resolution by condition (vi) of the previous corollary, but $\CC(Q)$ is not metrizable.

\begin{corollary} \label{c:Ck-strong-dual-G-fundam}
Let $X$ be such that $\CC(X)$ is locally complete. If $X$ has an increasing sequence of functionally bounded subsets which swallows the compact sets of $X$, then  the strong dual of $\CC(X)$ has a $\GG$-base.
\end{corollary}

\begin{proof}
By Theorem 3.1(ii) of \cite{FGK}, there is a metrizable locally convex topology $\TTT$ on $C(X)$ stronger than the compact-open topology, and hence $(C(X),\TTT)$ has a  fundamental bounded resolution by Corollary \ref{c:metrizable-fbr}. So $\CC(X)$ has a bounded resolution and Corollary \ref{c:Ck-strong-dual-G} applies.
\end{proof}

Recall that a lcs $E$ is called \emph{quasibarrelled} if every closed absolutely convex bornivorous subset of $E$ is a neigbourhood of zero, see \cite{bonet} or \cite{Smolyanov}, \cite{Jar}. Trivially, every metrizable lcs, as well as, any $(LM)$-space is quasibarrelled.

We  supplement  Theorem \ref{t:G-base-strong} with the following fact. Recall that $E$ is \emph{dual locally complete } if $(E',\sigma(E',E))$ is locally complete, see \cite{saxon}. Note that a lcs $E$ is barrelled if and only if $E$ is a quasibarelled dual locally complete space, \cite{bonet}.
\begin{proposition} \label{p:quasibarrelled-strong-dual-res}
The following statements hold true.
\begin{enumerate}
\item[{\rm (i)}]  Let $E$ be dual locally complete. If $E$ has a $\GG$-base, then $E'_\beta$ has a fundamental bounded resolution.
\item[{\rm (ii)}]  Let $E$ be a quasibarrelled space. Then the strong dual $E'_\beta$ of $E$ has a fundamental bounded resolution if and only if $E$ has a $\GG$-base.
\end{enumerate}
\end{proposition}

\begin{proof}
(i) Let $\{ U_\alpha: \alpha\in\NN^\NN\}$ be a $\GG$-base in $E$. Then the polar sets $W_\alpha:=U_\alpha^\circ$ are weakly$^{\ast}$-compact and absolutely convex, hence  Banach discs, so $(E',\sigma(E',E))$ is  a quasi-$(LB)$-space. Again Valdivia's  theorem \cite[Theorem~3.5]{kak} applies to get that $(E',\sigma(E',E))$ has a fundamental resolution $\mathcal{B}=\{ B_\alpha: \alpha\in\NN^\NN\}$  consisting of Banach discs. As  $(E',\sigma(E',E))$ is locally complete, every $\sigma(E',E)$-bounded set $B$ is included in a Banach disc by Proposition 5.1.6 of \cite{bonet}. Since Banach discs also are $\beta(E',E)$-bounded, the family  $\mathcal{B}$ is a  fundamental bounded resolution.

(ii) Assume that $E'_\beta$ has a fundamental bounded resolution. Then the strong bidual space $E''_\beta$ of $E$  has a $\GG$-base by  Theorem \ref{t:G-base-strong}. Since $E$ is quasibarrelled, $E$ is a subspace of $E''_\beta$ by Theorem 15.2.3 of \cite{NaB}. Therefore $E$ has a $\GG$-base. Conversely, assume that $E$ has a $\GG$-base $\{ U_\alpha: \alpha\in\NN^\NN\}$. Then the polar sets $W_\alpha:=U_\alpha^\circ$ are weakly$^{*}$-compact and absolutely convex, and hence $W_\alpha$ are $\beta(E',E)$-bounded by Theorem 11.11.5 of \cite{NaB}. To show that $\{ W_\alpha: \alpha\in\NN^\NN\}$ is a fundamental bounded resolution in $E'$, fix a strongly bounded subset $B$ of $E'$. Since $E$ is quasibarrelled, $W_\alpha$ is equicontinuous by Theorem 11.11.4 of \cite{NaB}. So there is $\alpha\in\NN^\NN$ such that $B\subseteq W_\alpha$.
\end{proof}

A subset $A$ of a Tychonoff space $X$ is {\em $b$-bounding} if for every bounded subset $B$ of $\CC(X)$ the number $\sup\{ |f(x)|: x\in A, f\in B\}$ is finite. The space $X$ is called a {\em $W$-space} if every $b$-bounding subset of $X$ is relatively compact.
\begin{corollary} \label{c:strong-dual-Ck-fund-res}
Let $X$ be a $W$-space (for example, $X$ is realcompact). Then the strong dual $E$ of $\CC(X)$ has a fundamental bounded resolution if and only if $X$ has a fundamental compact resolution.
\end{corollary}
\begin{proof}
Note that $\CC(X)$ is quasibarrelled by Theorem 10.1.21 of \cite{bonet}. Therefore, by Proposition \ref{p:quasibarrelled-strong-dual-res}, $E$ has a fundamental bounded resolution if and only if the space $\CC(X)$ has a $\GG$-base. But $\CC(X)$ has a $\GG$-base if and only if $X$ has a fundamental compact resolution by \cite{feka}.
\end{proof}

By an $(LM)$-space $E:=(E,\tau)$ we mean a lcs which is the  countably inductive limit of an increasing sequence $(E_{n},\tau_{n})$ of metrizable lcs such that $E=\bigcup_{n}E_{n}$  and $\tau_{n}|E_{n}\leq \tau_{n}$ for each $n\in\mathbb{N}$. The inductive limit topology $\tau$ of $E$ is the finest locally convex topology on $E$ such that $\tau|E_{n}\leq\tau_{n}$ for each $n\in\mathbb{N}$.  If each step $(E_{n},\tau_{n})$ is a Fr\'echet lcs, i.e. a metrizable and complete lcs, we call $E$ an $(LF)$-space. Moreover, if additionally  $\tau_{n+1}|E_{n}=\tau_{n}$ for each $n\in\mathbb{N}$ the inductive limit space $E$ is called \emph{strict}.  The latter case implies that every bounded set in $E$ is contained and bounded in some $E_{n}$. Recall that $(LM)$-spaces  enjoying this property are called \emph{regular.} We refer the reader to \cite[Definition 8.5.11]{bonet}  or to \cite{Jar} for details.

Next Proposition \ref{p:dual-regular-LM} provides fundamental bounded resolutions  for regular $(LM)$-spaces.

\begin{proposition} \label{p:dual-regular-LM}
Let $E$ be an  $(LM)$-space. Then $E$ has a  bounded resolution. If in addition $E$ is regular, then $E$ has a fundamental bounded resolution, and consequently, $E^{\prime}_{\beta}$ has a $\mathfrak{G}$-base.
\end{proposition}

\begin{proof}
Let $(E_{i},\tau _{i})$ be an increasing sequence of metrizable lcs generating the inductive limit space $E=(E,\tau )$, i.\thinspace e., $\tau _{i+1}|_{E_{i}}\leq \tau _{i}$ and $\tau |_{E_{i}}\leq \tau _{i}$ for any $i\in \mathbb{N}$.
For each $i\in \mathbb{N}$, let $\{ U_{k}^{i}\}_{k\in\NN}$ be a decreasing base of neighbourhoods of zero for $E_{i}$. For every $i\in \mathbb{N}$ and $\alpha \in \mathbb{N}^{\mathbb{N}}$, set $W_{\alpha }^{i}:=\bigcap_{k\in\NN }\alpha(k)U^{i}_{k}$. Any $W_{\alpha }^{i}$ is bounded in $\tau _{i}$, consequently in $\tau $ too. It is easy to see that the family $\{ W_\alpha^i : \alpha\in\NN^\NN\}$ is a fundamental bounded resolution in $E_i$.
Now, for every $\alpha=\big( \alpha(k)\big)\in\NN^\NN$, set $\alpha^\ast := \big( \alpha(k+1)\big)$ and  $B_\alpha:= W_{\alpha^\ast}^{\alpha(1)}$. Clearly, the family $\{ B_\alpha: \alpha\in\NN^\NN\}$ is a desired bounded resolution in $E$.

Assume that $E$ is regular. Then every $\tau $-bounded set is contained in some $E_{i}$ and is $\tau _{i}$-bounded. Therefore the bounded resolution $\{ B_\alpha: \alpha\in\NN^\NN\}$ is fundamental. Finally, the space  $E^{\prime}_{\beta}$ has a $\mathfrak{G}$-base by Theorem \ref{t:G-base-strong}.
\end{proof}

\begin{corollary} \label{c:image-metrizable-dual}
Let $E$ be a locally complete lcs which is an image of an infinite-dimensional metrizable topological vector space under a continuous linear map. Then every precompact set in $E_{\beta }^{\prime }$ is metrizable.
\end{corollary}

\begin{proof}
Let $T$  be a continuous linear map from a metrizable tvs $H$ onto $E$. It is well-known that $H$ has a bounded resolution $\{ B_\alpha: \alpha\in\NN^\NN\}$ (cf. the proof of Proposition \ref{p:LCS-bounded-resol}). Then $\{ T(B_\alpha): \alpha\in\NN^\NN\}$ is a bounded resolution on $E$. Now Theorem \ref{t:G-base-strong} implies that $E'_{\beta}$ has a $\GG$-base. Therefore every precompact set in $E'_{\beta}$ is metrtizable by Cascales--Orihuela's theorem, see \cite{kak}.
\end{proof}

%In what follows we shall use the following notations.
%For every $\alpha\in\NN^\NN$ and each $k\in\NN$, set
%\[
%I_k(\alpha):=\{ \beta\in\NN^\NN: \; \beta(1)=\alpha(1),\dots,\beta(k)=\alpha(k) \}.
%\]
%Let $\mathcal{B}=\{ B_\alpha: \alpha\in\NN^\NN\}$ be a resolution in a set $\Omega$. For  every $\alpha\in\NN^\NN$ and each $k\in\NN$, set
%\[
%A_k^\mathcal{B}(\alpha) := \bigcup_{\beta\in I_k(\alpha)} B_\beta.
%\]
%It is easy to see that the family $\AAA^\mathcal{B}:=\{ A_k^\mathcal{B}(\alpha): k\in\NN, \alpha\in\NN^\NN\}$ is countable. Moreover, for every $k\in \mathbb{N}$ and each $\alpha\leq\beta$ in $\NN^\NN$, we have
%\begin{equation} \label{equ:count-resolution}
%A_{k}^{\mathcal{B}}\left( \alpha \right) \subseteq A_{k}^{\mathcal{B}}\left( \beta \right).
%\end{equation}
%Indeed, if $x\in A_{k}^{\mathcal{B}}\left( \alpha \right) $, then there is $\gamma \in I_{k}\left( \alpha \right) $ with $x\in B_{\gamma }$, so that $\gamma %\left( i\right) =\alpha \left( i\right) $ for $1\leq i\leq k$. Now define $\delta \in \mathbb{N}^{\mathbb{N}}$ such that $\delta \left( i\right)  =\gamma %\left( i\right) $ for $i>k$ and $\delta \left( i\right) =\beta \left( i\right) $ for $1\leq i\leq k$. Since clearly $\gamma \leq \delta $, we have
% $B_{\gamma }\subseteq B_{\delta }$ and consequently $x\in B_{\delta }$. As in addition, by construction $\delta \in I_{k}\left( \beta \right) $, it
%turns out that $x\in A_{k}\left( \beta \right) $.

\begin{remark} {\em
Analysing the proof of Proposition 1 in \cite{feka-BAMS} one can prove the following result: A lcs $E$ has a bounded resolution if and only if $E_w$ is $K$-analytic-framed in $\IR^\kappa$ for some cardinal $\kappa$, that is there exists a $K$-analytic space $H$ such that $E\subseteq H\subseteq\IR^\kappa$.}
\end{remark}

\begin{remark}{\em
Christensen's theorem \cite[Theorem~6.1]{kak}  states that a metrizable space $X$ has a fundamental compact resolution if and only if $X$ is Polish. Therefore every {\em separable}  infinite-dimensional Banach space $E$ has a fundamental compact  resolution, and this resolution  is not a fundamental bounded resolution (otherwise, the closed unit ball $B$ of $E$ would be compact). On the other hand, every {\em nonseparable } infinite-dimensional metrizable space $E$ has a fundamental {\em bounded } resolution by Corollary \ref{c:metrizable-fbr}, but $E$ does not have a fundamental {\em compact} resolution since the space $E$ is not Polish. }
\end{remark}

%%%%%%%%%%%%%%%%%%%%%%%%%%%%%%%%%%%%%%%%%%
%%%%%%%%%%%%%%%%%%%%%%%%%%%%%%%%%%%%%%%%%%
%%%%%%%%%%%%%%%%%%%%%%%%%%%%%%%%%%%%%%%%%%
%%%%%%%%%%%%%%%%%%%%%%%%%%%%%%%%%%%%%%%%%%
%%%%%%%%%%%%%%%%%%%%%%%%%%%%%%%%%%%%%%%%%%

\section{More about [fundamental] bounded resolutions for spaces $C_{p}(X)$ and $\CC(X)$ and their duals} \label{sec:function}

%%%%%%%%%%%%%%%%%%%%%%%%%%%%%%%%%%%%%%%%%%
%%%%%%%%%%%%%%%%%%%%%%%%%%%%%%%%%%%%%%%%%%
%%%%%%%%%%%%%%%%%%%%%%%%%%%%%%%%%%%%%%%%%%
%%%%%%%%%%%%%%%%%%%%%%%%%%%%%%%%%%%%%%%%%%
%%%%%%%%%%%%%%%%%%%%%%%%%%%%%%%%%%%%%%%%%%
%%%%%%%%%%%%%%%%%%%%%%%%%%%%%%%%%%%%%%%%%%

Tkachuk proved (see  \cite[Theorem~9.14]{kak}) that the space $C_p(X)$ has a  fundamental compact resolution if and only if  $X$ is countable and discrete.  Hence, if for an infinite compact  space $K$, the space $C_p(K)$ is $K$-analytic, then $C_p(K)$ has a compact resolution but it does not have a  fundamental compact resolution. Also, if we consider the case $C\big([0,1]\big)$, then $C_w\big([0,1]\big)$ does not have a  fundamental compact resolution, see  \cite[Corollary 1.10]{MerSta}. Below we prove an analogous result for $C_{p}(X)$  having a fundamental bounded resolution. %We need the following notion.

We need the following notion. For every $\alpha\in\NN^\NN$ and each $k\in\NN$, set
\[
I_k(\alpha):=\{ \beta\in\NN^\NN: \; \beta(1)=\alpha(1),\dots,\beta(k)=\alpha(k) \}.
\]

\begin{definition} {\em
A family $\left\{ U_{\alpha ,n}:\left( \alpha ,n\right) \in \mathbb{N}^{\mathbb{N}}\times \mathbb{N}\right\} $ of closed subsets of $X$ is called {\em framing} if
\begin{enumerate}
\item [(1)] for each $\alpha \in \mathbb{N}^{\mathbb{N}}$, the family $\left\{ U_{\alpha ,n}:n\in \mathbb{N}\right\} $ is an increasing covering of $X$, and
\item [(2)] for every $n\in \NN$, $U_{\beta ,n}\subseteq U_{\alpha ,n}$ whenever $\alpha \leq \beta $.
\end{enumerate} }
\end{definition}

Recall that a family $\mathcal{N}$ of subsets of a topological space $X$ is called a  {\em $cs^\ast$-network at  a point} $x\in X$ if for each sequence $\{ x_n\}_{n\in\NN}$ in $X$ converging to  $x$ and for each neighborhood $O_x$ of $x$ there is a set $N\in\mathcal{N}$ such that $x\in N\subseteq O_x$ and the set $\{n\in\NN :x_n\in N\}$ is infinite; $\Nn$ is a {\em $cs^\ast$-network}  in $X$ if $\mathcal{N}$ is a $cs^\ast$-network at each point $x\in X$. % see \cite{GKL2}.

Now we are ready to prove  the main result of this section.
\begin{theorem}\label{t:FBR-Cp}
The space $C_{p}(X)$ admits  a fundamental bounded resolution if and only if $X$ is countable (so exactly when $C_{p}(X)$ is metrizable).
\end{theorem}

\begin{proof}%[Proof of Theorem \ref{t:FBR-Cp}]
If $X$ is countable, the space $C_p(X)$  has a fundamental bounded resolution by Corollary \ref{c:metrizable-fbr}. Conversely, assume that $C_{p}\left( X\right) $ has a fundamental bounded resolution $\left\{ B_{\alpha }:\alpha \in \mathbb{N}^{\mathbb{N}}\right\} $. We prove that $X$ is countable in five steps.
First we note the   following simple observation.

\medskip
{\em Step 1. A subset $Q$ of $C_p\left( X\right) $ is bounded if and only if there exists an increasing covering $\left\{ V_{n}:n\in \mathbb{N}\right\} $ of $X$ consisting of closed sets such that
\[
\sup_{f\in Q}\left| f\left( x\right) \right| \leq n \mbox{ for all } x\in V_{n}.
\] }

\smallskip
Indeed, assume that $Q$ is a bounded subset of $C_p(X)$. For every $n\in \mathbb{N}$, set
\begin{equation*}
V_{n}=\left\{ x\in X:\sup\nolimits_{f\in Q}\left| f\left( x\right) \right| \leq n\right\}.
\end{equation*}%
Clearly, all $V_n$ are closed, $V_{n}\subseteq V_{n+1}$ for each $n\in \mathbb{N}$, and $\sup_{f\in Q}\left| f\left( x\right) \right| \leq n$ for all $x\in V_{n}$. If $y\in X$ there is $m\in \mathbb{N}$ with $\sup_{f\in Q}\left| f\left( y\right) \right| \leq m$, so that $y\in V_{m}$. This shows that $\bigcup_{n\in\NN} V_{n}=X$.

The converse assertion is clear.

\medskip
{\em Step 2. There exists a framing family $\left\{ U_{\alpha ,n}:\left( \alpha ,n\right) \in \mathbb{N}^{\mathbb{N}}\times \mathbb{N} \right\} $ in $X$ enjoying the property that if $\left\{ V_{n}:n\in \mathbb{N}\right\} $ is an increasing covering of $X$ consisting of closed sets there exists $\gamma \in \mathbb{N}^{\mathbb{N}}$ such that $U_{\gamma,n}\subseteq V_{n}$ for all $n\in \mathbb{N}$.}

\smallskip
Indeed, for each $\alpha\in\NN^\NN$ and every $n\in\NN$, set
\begin{equation*}
U_{\alpha ,n}=\left\{ x\in X:\sup\nolimits_{f\in B_{\alpha }}\left| f\left( x\right) \right| \leq n\right\}.
\end{equation*}
Then, for each $\alpha \in \mathbb{N}^{\mathbb{N}}$, the family $\left\{ U_{\alpha ,n}:n\in \mathbb{N}\right\} $ is an increasing closed covering of $X$
such that $U_{\beta ,n}\subseteq U_{\alpha ,n}$ whenever $\alpha \leq \beta $, and in addition $\sup\nolimits_{f\in B_{\alpha }}\left| f\left( x\right) \right| \leq n$ for every $x\in U_{\alpha ,n}$ and $n\in \mathbb{N}$. Therefore the family $\mathcal{U}:=\left\{ V_{\alpha ,n}:\left( \alpha ,n\right) \in \mathbb{N}^{\mathbb{N}}\times \mathbb{N}\right\} $ is framing.

We claim that $\mathcal{U}$ satisfies the stated property. Indeed, fix an increasing covering  $\left\{ V_{n}:n\in \mathbb{N}\right\} $ of $X$ consisting of closed sets. Set
\begin{equation*}
P:=\{f\in C\left( X\right) :\sup\nolimits_{x\in V_{n}}\left| f\left( x\right) \right| \leq n\ \forall n\in \mathbb{N}\}.
\end{equation*}%
Then $P$ is a bounded subset of $C_p\left( X\right) $ by  Step 1. Since $\left\{ B_{\alpha }:\alpha \in \mathbb{N}^{\mathbb{N}}\right\}
$ is a fundamental bounded resolution for $C_{p}\left( X\right) $, there exists $\delta \in \mathbb{N}^{\mathbb{N}}$ such that $P\subseteq B_{\delta
} $. To prove the claim we show that $U_{\delta ,n}\subseteq V_n$ for every $n\in\NN$. Take arbitrarily $x\in U_{\delta ,n}$. Then
\[
\sup\nolimits_{f\in P}\left| f\left( x\right) \right| \leq  \sup\nolimits_{f\in B_{\delta }}\left| f\left( x\right) \right| \leq n.
\]
Now  if $x\notin \overline{V_{n}}=V_{n}$, there is $h\in C\left( X\right) $ with $0\leq h\leq n+1$ such that $h\left( x\right) =n+1$ and $h\left( y\right) =0$ for every $y\in V_{n}$. By construction of $h$ and since $\{ V_n\}_n$ is increasing, we have $|h(y)|\leq n$ for every $n\in\NN$ and each $y\in V_n$. Therefore  $h\in P\subseteq B_{\delta }$. So we have at the same time that $x\in U_{\delta ,n}$ and $\left| h\left( x\right) \right| =n+1$ with $h\in B_{\delta }$, a contradiction. Thus $U_{\delta ,n}\subseteq V_{n}$ for every $n\in \mathbb{N}$.

\medskip
{\em Step 3. For every $n\in\NN$ and each $\alpha\in\NN^\NN$, set $M_n(\alpha):= \bigcup_{\beta\in I_n(\alpha)} B_\beta$ and
\[
K_n(\alpha) :=\bigcap_{\beta\in I_n(\alpha)} U_{\beta,n} =\left\{ x\in X: \sup_{f\in M_n(\alpha)} |f(x)| \leq n\right\}.
\]
Then all $K_n(\alpha)$ are closed (but can be empty) and satisfy the following conditions:
\begin{enumerate}
\item[(i)]  $K_n(\alpha) \subseteq K_{n+1}(\alpha)$ for every $n\in\NN$ and each $\alpha\in\NN^\NN$;
\item[(ii)]  $K_n(\alpha) \supseteq K_{n}(\beta)$ for every $n\in\NN$ whenever $\alpha\leq\beta$;
\item[(iii)]  $\bigcup_{n\in\NN} K_n(\alpha)=X$ for each $\alpha\in\NN^\NN$;
\item[(iv)] for every increasing closed covering $\left\{ V_{n}:n\in \mathbb{N}\right\} $ of $X$ there exists $\gamma \in \mathbb{N}^{\mathbb{N}}$ such that $K_n (\gamma)\subseteq V_{n}$ for all $n\in \mathbb{N}$.
\end{enumerate}
Moreover, the family $\KK:=\{ K_n(\alpha) : \; n\in\NN, \alpha\in\NN^\NN\}$ is countable. }

\smallskip
Indeed, (i) and (ii) are clear. To prove (iii) suppose for a contradiction that there is $x\not\in \bigcup_{n\in\NN} K_n(\alpha)$ for some $\alpha\in\NN^\NN$. For every $n\in\NN$ choose $\beta_n\in I_n(\alpha)$ such that $x\not\in U_{\beta_n,n}$. Set $\gamma:= \sup\{ \beta_n :n\in\NN\}$. Then, for every $n\in\NN$, $\beta_n\leq \gamma$ and hence $x\not\in U_{\gamma,n}$ since $U_{\gamma,n} \subseteq U_{\beta_n,n}$ by the definition of a framing family. Therefore $x\not\in \bigcup_{n\in\NN} U_{\gamma,n} =X$, a contradiction. Thus (iii) holds. Now we prove (iv). By the condition of the theorem, there is $\gamma\in \NN^\NN$ such that $U_{\gamma,n} \subseteq V_n$ for all $n\in\NN$. Then $K_n (\gamma)\subseteq V_{n}$ since $K_n (\gamma) \subseteq U_{\gamma,n} $ for all $n\in\NN$.
Finally, the family $\KK$ is countable since, by construction, the set $K_n(\alpha)$ depends only on $\alpha(1),\dots,\alpha(n)$.

\medskip
{\em Step 4. For every $m,n\in\NN$ and each $\alpha\in\NN^\NN$, set
\[
N_{mn}(\alpha) := \left\{ f\in C(X): |f(x)|\leq \frac{1}{m} \;\; \forall  x\in K_n(\alpha) \right\}
\]
(if $K_n(\alpha)$ is empty we set $N_{mn}(\alpha) :=\{ 0\}$).
We claim that the  family
\[
\Nn := \left\{ N_{mn}(\alpha): \; m,n\in\NN \mbox{ and } \alpha\in\NN^\NN \right\}
\]
is a countable $cs^\ast$-network at $0\in C_p(X)$.}

\smallskip
Indeed, the family $\Nn$ is countable since the family $\KK$ is countable. To show that $\Nn$ is a $cs^\ast$-network at $0\in C_p(X)$, let $S=\{ g_n: n\in\NN\}$ be a null-sequence in $C_p(X)$ and  let $U$ be a standard neighborhood of zero in $C_p(X)$ of the form
\[
U=[F,\e]:= \{ f\in C(X): |f(x)|<\e \;\; \forall x\in F \},
\]
where $F$ is a finite subset of $X$ and $\e>0$. Fix arbitrarily an $m\in\NN$ such that $m>1/\e$. For every $n\in\NN$, set
\[
T_{n}:= \bigcap_{i\geq n} R_{i}, \; \mbox{ where } \; R_{i}:= \left\{ x\in X: \; |g_i(x)| \leq \frac{1}{m} \right\}.
\]
It is clear that $\{ T_{n}\}_{n\in\NN}$ is an increasing sequence of closed subsets of $X$. Moreover, since $g_n\to 0$ in $C_p(X)$ we obtain that $\bigcup_{n\in\NN} T_n =X$. Therefore, by (iv), there is a $\gamma\in\NN^\NN$ such that
\begin{equation} \label{equ:Cp-FBR-1}
K_n(\gamma) \subseteq T_n \; \mbox{ for every } n\in\NN.
\end{equation}
Now, by (iii), choose an $n\in\NN$ such that $F\subseteq K_n(\gamma)$. Then (\ref{equ:Cp-FBR-1}) implies
\[
\{ g_i\}_{i\geq n} \subseteq N_{mn}(\gamma) \subseteq U=[F,\e].
\]
Thus  $\Nn$ is a countable $cs^\ast$-network at zero.

\medskip
{\em Step 5. The space $X$ is countable}. Indeed, since the space $C_p(X)$ has a countable $cs^\ast$-network at zero by Step 4, the space $X$ is countable by \cite{Sak} (or \cite{GaK}, recall that $C_p(X)$ is $b$-Baire-like for every Tychonoff space $X$).
\end{proof}

\begin{example} \label{exa:Cp-no-brs} {\em
Let $X$ be an uncountable pseudocompact space. Then the space $C_p(X)$ has a bounded resolution but it does not have a fundamental bounded resolution.
Indeed, as $X$ is pseudocompact,  the sets
$
B_n=\{ f\in C(X): |f(x)|\leq n \; \forall x\in X\}
$
form a bounded resolution (even a bounded sequence) in $C_p(X)$. The second assertion follows from Theorem \ref{t:FBR-Cp}.}
\end{example}

For Lindel\"{o}f $P$-spaces  we obtain even a stronger result.
\begin{proposition} \label{p:Cp-bounded-res-P-space}
Let $X$ be a Lindel\"{o}f $P$-space. Then $C_p(X)$ has a bounded resolution if and only if $X$ is countable and discrete.
\end{proposition}

\begin{proof}
Assume that $C_p(X)$ has a bounded resolution. Then $C_p(X)$ is  sequentially complete (equivalently locally complete) by \cite{FKS}. As $X$ is a Lindel\"{o}f $P$-space, the space $C_p(X)$ is Frech\'{e}t--Urysohn by Theorem~II.7.15 of \cite{Arhangel}. By Theorem 14.1 of \cite{kak}, every sequentially complete  Frech\'{e}t--Urysohn lcs is Baire, so $C_p(X)$ is a Baire space.  But any Baire topological vector space  with a bounded resolution is metrizable by Proposition 7.1 of \cite{kak}. Therefore $C_p(X)$ is metrizable. Thus $X$ is countable. The converse assertion follows from Corollary \ref{c:metrizable-fbr} and the trivial fact that every countable $P$-space is discrete.
\end{proof}

Recall (see \cite{Mich}) that a Tychonoff space $X$ is called a {\em cosmic space} (an {\em $\aleph_{0}$-space}) if $X$ is an  image of a separable metric space under a continuous (respectively, compact-covering) map.

\begin{theorem} \label{t:strong-dual-G-cosmic}
Let  $E=\CC(\CC(X))$. Then:
\begin{enumerate}
\item[{\rm (i)}] If $X$ is metrizable, then $E$ has  a $\GG$-base if and only if $X$ is $\sigma$-compact. In this case $E$ is barrelled.
\item[{\rm (ii)}] If $X$ is an  $\aleph_0$-space, then the strong dual $E'_\beta$ of $E$ has a $\GG$-base if and only if $X$ is finite. In particular, if $X$ is infinite, then $E$ is not a regular $(LM)$-space.
\item[{\rm (iii)}] If $X$ is a $\mu$-space and $E'_\beta$ has a $\GG$-base, then $X$ has a  fundamental compact resolution, so that $\CC(X)$ has a $\GG$-base. But the converse is not true in general.
\end{enumerate}
Consequently, if $X$ is an infinite metrizable  $\sigma$-compact space, then $E$ is a barrelled space with a $\GG$-base whose strong dual $E'_\beta$ does not have a $\GG$-base.
\end{theorem}

\begin{proof}
(i) If $E$ has a $\GG$-base, then $\CC(X)$ has a fundamental  compact resolution by Theorem 2 of \cite{feka}. Therefore $X$ is $\sigma$-compact by Corollary 9.2 of \cite{kak}. Conversely, if $X$ is $\sigma$-compact, then $\CC(X)$ has a  fundamental compact resolution by Corollary 2.10 of \cite{GK-Free-LCS-G}. Once again applying Theorem 2 of \cite{feka}, we obtain that $E$ has a $\GG$-base.

To prove the last assertion we note that any metrizable $\sigma$-compact space is an $\aleph_0$-space, and hence $\CC(X)$ is Lindel\"{o}f by \cite{Mich}. So $\CC(X)$ is a $\mu$-space and the space $E$ is barrelled by the Nachbin--Shirota theorem.

(ii) Assume that the strong dual of $\CC(\CC(X))$ has a $\GG$-base. Then, by Theorem \ref{t:G-base-strong}, the space $\CC(\CC(X))$ and hence also $C_p(\CC(X))$ have a bounded resolution. Since $\CC(X)$ is also cosmic by Proposition 10.3 of \cite{Mich}, Corollary 9.1 of \cite{kak} implies that the space $\CC(X)$ is $\sigma$-compact. Therefore $C_p(X)$ is also $\sigma$-compact. Now Velichko's theorem \cite[Theorem~9.12]{kak} implies that $X$ is finite. Conversely, if $X$ is finite and $|X|=n$, then $E=\CC(\IR^n)$ is a Fr\'{e}chet space and Theorem \ref{t:G-base-strong} applies.

The last assertion follows from this result and Proposition \ref{p:dual-regular-LM}.

(iii) Let $M_c(X)$ be the topological dual of $\CC(X)$. Denote by $\TTT_k$ and $\TTT_\beta$ the compact-open topology induced from $E$ and the strong topology $\beta(M_c(X),C(X))$ on $M_c(X)$, respectively. Clearly, $\TTT_k \leq \TTT_\beta$. Set $G:= (M_c(X),\TTT_k)'$ and  $F:= (M_c(X),\TTT_\beta)'$. Then $C(X) \subseteq G\subseteq F$, algebraically.  Hence every $\sigma(M_c(X),F)$-bounded subset of $M_c(X)$ is also $\sigma(M_c(X),G)$-bounded, and therefore
\begin{equation} \label{equ:strong-dual-G-1}
\beta\big(F,M_c(X)\big)|_G \leq \beta\big(G,M_c(X)\big).
\end{equation}

Observe that $\CC(X)$ is barrelled by the Nachbin--Shirota theorem, and hence, by Theorem 15.2.3 of \cite{NaB}, the space $\CC(X)$ is a subspace of its strong bidual space $\big( F, \beta\big(F,M_c(X)\big)\big)$. Therefore
\begin{equation} \label{equ:strong-dual-G-2}
\tau_k =\beta\big(F,M_c(X)\big)|_{C(X)}.
\end{equation}

As $(M_c(X),\TTT_k)$ is a closed subspace of $E$ we obtain $G= (M_c(X),\TTT_k)'= E'/M_c(X)^\perp$, algebraically. Denote by $\TTT_q$ the quotient topology of the strong dual $E'_\beta$ of $E$ on $G$. We claim that the strong topology $\beta(G,M_c(X))$ on $G$ is coarser than $\TTT_q$. Indeed, if $j$ is the canonical inclusion of $(M_c(X),\TTT_k)$ into $E$, the adjoint map $j^\ast$ is strongly continuous, see \cite[Theorem~8.11.3]{NaB}. Then the claim follows from the fact that $j^\ast$ is raised to the continuous map from the quotient $E'_\beta/M_c(X)^\perp$ to the strong dual of $(M_c(X),\TTT_k)$. Now the claim and (\ref{equ:strong-dual-G-1}) and (\ref{equ:strong-dual-G-2}) imply
\begin{equation} \label{equ:strong-dual-G-3}
\tau_k =\beta\big(F,M_c(X)\big)|_{C(X)}  \leq \beta\big(G,M_c(X)\big)|_{C(X)} \leq \TTT_q|_{C(X)}.
\end{equation}

Suppose for a contradiction that $E'_\beta$ has a $\GG$-base. Then the quotient topology $\TTT_q$ and hence $\TTT_q|_{C(X)}$ also have a $\GG$-base. %, see Proposition 2.7 of \cite{GKL}.
Therefore, by (\ref{equ:strong-dual-G-3}), there exists a locally convex topology $\TTT:=\TTT_q|_{C(X)}$ with a $\GG$-base on $C(X)$ such that $\tau_k\leq \TTT$. According to \cite[Corollary~2.3]{FGK} applied to the family $\mathcal{S}=\mathcal{K}(X)$ of all compact subsets of $X$ and $C(X)$, %\cite[Theorem~3.1(vi)]{FGK}, where the inequality $\TTT\leq\tau_b$ is not necessary for the 'only if' case,
we obtain that $X$ has a functionally bounded resolution swallowing the compact sets of $X$. Finally, since $X$ is a $\mu$-space, it follows that $X$ admits a  fundamental compact resolution.

Observe that for  $X=\IR$ the space  $X$ is even hemicompact, but $E'_\beta$ does not have a $\GG$-base by (ii).
\end{proof}

%%%%%%%%%%%%%%%%%%%%%%%%%%%%%%%%%%%%%%%%
%%%%%%%%%%%%%%%%%%%%%%%%%%%%%%%%%%%%%%%%
%%%%%%%%%%%%%%%%%%%%%%%%%%%%%%%%%%%%%%%%
%%%%%%%%%%%%%%%%%%%%%%%%%%%%%%%%%%%%%%%%

Following Markov \cite{Mar}, the {\em  free lcs} $L(X)$ over a Tychonoff space $X$ is a pair consisting of a lcs $L(X)$ and  a continuous mapping $i: X\to L(X)$ such that every  continuous mapping $f$ from $X$ to a lcs $E$ gives rise to a unique continuous linear operator ${\bar f}: L(X) \to E$  with $f={\bar f} \circ i$. The free lcs $L(X)$  always exists and is  unique.
The set $X$ forms a Hamel basis for $L(X)$, and  the mapping $i$ is a topological embedding. % \cite{Rai,Usp2}.
Denote by $L_p(X)$ the free lcs $L(X)$ endowed with the weak topology.

Following \cite{BG}, a Tychonoff space $X$  is called an {\em Ascoli space} if every compact subset $\KK$ of $\CC(X)$  is equicontinuous (see \cite{Gabr-LCS-Ascoli}). By Ascoli's theorem \cite{NaB}, each $k$-space is Ascoli. The following theorem complements and extends Theorem 3.2 and Corollaries 3.3 and 3.4 of \cite{GK-Free-LCS-G}.
%\cite[Theorem 3.4.20]{Eng}

\begin{theorem} \label{t:Ascoli-free-G-base}
For an Ascoli space $X$ the following assertions are equivalent:
\begin{enumerate}
\item[{\rm (i)}] $L(X)$ has a $\GG$-base;
\item[{\rm (ii)}] $\CC(X)$ has  a  fundamental compact  resolution;
\item[{\rm (iii)}] $\CC\big(\CC(X)\big)$ has a $\GG$-base.
\end{enumerate}
In particular, any  item above implies that  every compact subset of $X$ is metrizable.
\end{theorem}

\begin{proof}
(i)$\Rightarrow$(ii) Let $\U=\{ U_\alpha: \alpha\in\NN^\NN\}$ be a $\GG$-base in $L(X)$. Then the family $\U^\circ=\{ U_\alpha^\circ: \alpha\in\NN^\NN\}$, where the polars are taken in the dual space $L(X)'$ of $L(X)$, is a compact resolution of $\big( L(X)', \sigma(L(X)',L(X))\big)$.  It is well known and easy to see that the dual space $L(X)'$ of $L(X)$ can be identified with the space $C(X)$ under the restriction map (recall that $X$ is a Hamel base for $L(X)$) %(\cite{Rai})
\[
L(X)' \ni \chi \mapsto \chi|_X \in C(X).
\]
It is clear that the weak* topology  $\sigma(L(X)',L(X))$ on $L(X)'$ induces the pointwise topology $\tau_p$ on $C(X)$. Therefore, for every $\alpha\in\NN^\NN$, the set $K_\alpha:=\{ \chi|_X: \chi\in U_\alpha^\circ \}$ is closed also in the compact-open topology $\tau_k$ on $C(X)$. We claim that $K_\alpha$ is a compact subset of $\CC(X)$. For this, by the Ascoli theorem \cite[Theorem~5.10.4]{NaB}, it is sufficient to check that $K_\alpha$ is pointwise bounded and equicontinuous. Clearly, $K_\alpha$ is pointwise bounded. To show that  $K_\alpha$ is equicontinuous fix an $x\in X$. Then for every $y=x+t \in (x+U_\alpha)\cap X$ (recall that $X$ is a subspace of $L(X)$), we obtain
\[
\big| \chi|_X (y) - \chi|_X(x)\big| = |\chi(t)|\leq 1, \quad \forall \, \chi|_X\in K_\alpha.
\]
Thus $K_\alpha$ is a compact subset of $\CC(X)$. Consequently, the family $\KK:=\{ K_\alpha: \alpha\in\NN^\NN\}$ is a compact resolution in $\CC(X)$.

Take arbitrarily a compact subset $K$ of $\CC(X)$. Then the polar $K^\circ$ of $K$ in $\CC(\CC(X))$ is a neighborhood of zero in $\CC(\CC(X))$. Since $X$ is Ascoli, the space $L(X)$ is a subspace of $\CC(\CC(X))$ by Theorem 1.2 of \cite{Gabr-LCS-Ascoli}. So there is $\alpha\in\NN^\NN$ such that $U_\alpha \subset K^\circ \cap L(X)$. Hence $K \subseteq K^{\circ\circ} \subseteq U_\alpha^\circ$. Thus $\KK$ swallows the compact sets of $\CC(X)$.

(ii)$\Rightarrow$(iii) follows from Theorem 2 of \cite{feka}, and (iii) implies (i) since $L(X)$ is a subspace of $\CC(\CC(X))$ by Theorem 1.2 of \cite{Gabr-LCS-Ascoli}.

The last assertion follows from the fact that $X$ is a subspace of $L(X)$ and Cascales--Orihuela's theorem \cite[Theorem 11]{CasOr} (which states that every compact subset of a lcs  with a $\GG$-base is metrizable).
\end{proof}
The previous theorem may suggest the following problem: Characterize in terms of $X$ those spaces $C_{k}(X)$ with a fundamental bounded resolution.

%\begin{problem}\label{ascoli}
%Characterize in terms of $X$ those spaces $C_{k}(X)$ with a fundamental bounded resolution.
%\end{problem}

%%%%%%%%%%%%%%%%%%%%%%%%%%%%%%%%%%%%%%%%%%
%%%%%%%%%%%%%%%%%%%%%%%%%%%%%%%%%%%%%%%%%%
%%%%%%%%%%%%%%%%%%%%%%%%%%%%%%%%%%%%%%%%%%
%%%%%%%%%%%%%%%%%%%%%%%%%%%%%%%%%%%%%%%%%%
We propose also the following

\begin{proposition} \label{p:strong-dual-Cp-fund-res}
The following assertions are equivalent:
\begin{enumerate}
\item [(i)]  The strong dual space $L(X)_\beta$ of $C_p(X)$ has  a bounded resolution.
\item [(ii)] $L(X)_\beta$ has  a fundamental bounded resolution.
\item [(iii)]   $X$ is countable.
\end{enumerate}
\end{proposition}
\begin{proof}
(i)$\Rightarrow$(iii) Assume that $L(X)_\beta$ has a bounded resolution $\{ B_\alpha: \alpha\in\NN^\NN\}$. Since $C_p(X)$ is quasibarrelled, its strong dual $L(X)_\beta$ is feral, see p. 392 of \cite{FKS-feral} (recall that following \cite{KST}, a lcs $E$ is called\emph{ feral } if every bounded set of $E$ is finite-dimensional). If $X$ has an uncountable number of points, some set $B_\alpha$ of the resolution would contain infinitely many points of $X$ by  \cite[Proposition 3.7]{kak}. Since $X$ is a Hamel basis for $L(X)$, the set $B_\alpha$ would be infinite-dimensional, a contradiction. Thus $X$ must be countable.

(iii)$\Rightarrow$(ii) If $X$ is countable, $C_p(X)$ is metrizable. Thus  $E$  has a fundamental bounded resolution by Proposition \ref{p:quasibarrelled-strong-dual-res}. Finally, the implication (ii)$\Rightarrow$(i) is trivial.
\end{proof}

The following item  characterizes those $\mu$-spaces $X$ for which  the weak$^{*}$ dual $L_{p}(X)$ of $C_{p}(X)$ has a fundamental bounded resolution.
\begin{proposition} \label{p:FBR-L(X)}
Let $X$ be a $\mu $-space. Then $X$ has a fundamental compact resolution if and only if $L_{p}\left( X\right) $ has a fundamental bounded resolution.
\end{proposition}

\begin{proof}
Denote by $\delta :X\rightarrow L_{p}\left( X\right) $  the canonical embedding map and let $E$ be the topological dual of $\CC\left( X\right) $. First we observe that if $X$ is a $\mu $-space, then $\CC\left( X\right) $ is the strong dual of $L_{p}\left( X\right) $. Indeed, since $X$ is a $\mu $-space,  the barrelledness of $\CC\left( X\right) $ yields $\tau_{k}=\beta \left( C\left( X\right) ,E\right) $. Since clearly $\tau_k\leq \beta \left( C\left( X\right) ,L\left( X\right) \right) \leq \beta \left( C\left( X\right) ,E\right) $, it follows that $\tau_k=\beta \left( C\left( X\right) ,L\left( X\right) \right) $.

We claim that if $A$ is a bounded subset of $L_{p}\left( X\right) $, there is a compact set $K$ in $X$ and $n\in\NN$ such that $A\subseteq  n\cdot \overline{\mathrm{abx}\left( \delta \left( K\right) \right) }^{L_{p}\left( X\right) }$. Indeed, by the previous observation, if $A$ is a bounded set in $L_{p}\left( X\right) $ there are $n\in \mathbb{N}$ and a compact subset $K$ of $X$ such that
\begin{equation*}
\left\{ f\in C\left( X\right) :\sup\nolimits_{x\in K}\left| f\left( x\right) \right| \leq n^{-1}\right\} \subseteq A^{\circ}.
\end{equation*}%
So, if $u_{f}$ denotes the (unique) continuous linear extension of $f$ to $L_{p}\left( X\right) $, we have
\begin{equation*}
\delta \left( K\right)^{\circ}=\left\{ f\in C\left( X\right) :\sup\nolimits_{x\in K}\left| \left\langle \delta _{x},u_{f}\right\rangle
\right| \leq 1\right\} \subseteq nA^{\circ}.
\end{equation*}%
Hence $A\subseteq A^{\circ\circ}\subseteq n\,\delta \left( K\right)^{\circ\circ}$, where the bipolar is taken in $E$. Thus $A\subseteq n\cdot\overline{\mathrm{abx}\left( \delta \left( K\right) \right) }^{L_{p}\left( X\right) }$.

Assume that $X$ has a fundamental compact resolution $\KK=\{K_{\alpha }:\alpha \in \mathbb{N}^{\mathbb{N}}\}$. For every $\alpha=\big( \alpha(i)\big) \in  \mathbb{N}^{\mathbb{N}}$, set $\alpha ^{\ast }:=\big( \alpha \left( i+1\right)\big) $ and
\[
B_{\alpha }:=\alpha \left( 1\right) \cdot \overline{\mathrm{abx}\left( \delta \left( K_{\alpha ^{\ast }}\right) \right) }^{L_{p}\left( X\right) }.
\]
Clearly the family $\mathcal{B}:= \{B_{\alpha }:\alpha \in \mathbb{N}^{\mathbb{N}}\}$ consists of bounded sets in $L_{p}\left( X\right) $ and satisfies that $B_{\alpha }\subseteq B_{\beta }$ whenever $\alpha \leq \beta $. Moreover, according to the claim and the fact that the family $\KK$ is fundamental, we obtain that $\mathcal{B}$ is a fundamental bounded resolution for $L_{p}\left( X\right) $.

Conversely, let  $\{A_{\alpha }:\alpha \in \mathbb{N}^{\mathbb{N}}\}$ be a fundamental bounded resolution for $L_{p}\left( X\right) $. For each $\alpha \in \NN^\NN$, set $M_{\alpha }:=A_{\alpha }\cap \delta \left( X\right) $.  Then $\{M_{\alpha }:\alpha \in \mathbb{N}^{\mathbb{N}}\}$ is a resolution for $\delta \left( X\right) $ consisting of functionally bounded sets that swallows the functionally bounded subsets of $\delta \left( X\right) $. Since $X$ is a $\mu $-space, then the family $\left\{ \overline{\delta ^{-1}\left( M_{\alpha }\right) }^{X}:\alpha \in \mathbb{N}^{\mathbb{N}}\right\}$ is a fundamental compact resolution for $X$.
\end{proof}

The condition that $X$ is  a $\mu $-space cannot be removed from Proposition \ref{p:FBR-L(X)}, as the following example shows.

\begin{example} {\em
Let $X$ be the ordinal space $\left[ 0,\omega _{1}\right) $, where $\omega_{1}$ is the first uncountable ordinal, and set $Y:=\left[ 0,\omega _{1}\right] $. Since $X$ is pseudocompact, the restriction map $Tf=\left. f\right| _{X}$ maps continuously $C_{p}\left( Y\right) $ onto $C_{p}\left( X\right) $. Consequently, $L_{p}\left( X\right) $ is topologically isomorphic to a linear subspace of $ L_{p}\left( Y\right) $. Since $Y$ is a compact set, according to Proposition  \ref{p:FBR-L(X)} the space $L_{p}\left( Y\right) $ has a (countable) fundamental bounded resolution, which implies that $L_{p}\left(X\right) $ also has a fundamental bounded resolution. However, under $MA+\lnot CH$ the space $X=\left[ 0,\omega _{1}\right) $ even does not have a compact resolution \cite[Theorem 3.6]{Tkachuk}. Observe that $X$ is not a $\mu $-space.}
\end{example}

Next example shows that there is no natural relationship between the existence of fundamental bounded resolutions for  different natural topologies.
\begin{example} {\em
If $X$ is an uncountable Polish space, then the strong dual of $C_p(X)$ does not have a fundamental bounded resolution by Proposition \ref{p:strong-dual-Cp-fund-res}, but the weak${}^\ast$ dual of $C_p(X)$ has a  fundamental bounded resolution by Proposition \ref{p:FBR-L(X)} and Christensen's theorem (see, \cite[Theorem~6.1]{kak}). On the other hand, if $X$ is a countable but non-Polish metrizable space (for example, $X=\mathbb{Q}$), then the strong dual of $C_p(X)$ has a fundamental bounded resolution by Proposition \ref{p:strong-dual-Cp-fund-res}, however the weak${}^\ast$ dual of $C_p(X)$ does not have a  fundamental bounded resolution by Proposition \ref{p:FBR-L(X)} and Christensen's theorem.}
\end{example}

\section{quasi-$(DF)$-spaces} \label{sec:quasi-DF}

%%%%%%%%%%%%%%%%%%%%%%%%%%%%%%%%%%%%%%%
%%%%%%%%%%%%%%%%%%%%%%%%%%%%%%%%%%%%%%%
%%%%%%%%%%%%%%%%%%%%%%%%%%%%%%%%%%%%%%%
%%%%%%%%%%%%%%%%%%%%%%%%%%%%%%%%%%%%%%%
%%%%%%%%%%%%%%%%%%%%%%%%%%%%%%%%%%%%%%%
In Introduction we formally defined the class of quasi-$(DF)$-spaces. The previous sections apply to gather a few fundamental properties of quasi-$(DF)$-spaces. We refer the readers to corresponding facts dealing with $(DF)$-spaces, see \cite{Jar} and \cite{bonet}. Note however that quasi-$(DF)$-spaces are stable by taking countable products although this property  fails for $(DF)$-spaces.
\begin{theorem} \label{t:properties-quasi-DF-spaces}
The following statements hold true.
\begin{enumerate}
\item [{\rm (i)}] Every  regular $(LM)$-space $E$ is a quasi-$(DF)$-space. In particular, every infinite-dimensional metrizable non-normable lcs $E$ is a quasi-$(DF)$-space not being a $(DF)$-space.
\item [{\rm (ii)}]   The strong dual $E'_\beta$ of  a regular $(LM)$-space is a quasi-$(DF)$-space.
\item [{\rm (iii)}] Countable direct sums of  quasi-$(DF)$-spaces are   quasi-$(DF)$-spaces.
\item [{\rm (iv)}] A subspace of a  quasi-$(DF)$-space is a  quasi-$(DF)$-space.
\item [{\rm (v)}] A countable product $E$ of  quasi-$(DF)$-spaces is a  quasi-$(DF)$-space.
\item [{\rm (vi)}] Every precompact set of a  quasi-$(DF)$-space is metrizable.
\end{enumerate}
\end{theorem}
\begin{proof}
(i) Every  $(LM)$-space belongs to the class $\GG$, see \cite[Section~11.1]{kak}. Being regular, $E$ has a fundamental bounded resolution by Proposition \ref{p:dual-regular-LM}. Thus $E$ is a quasi-$(DF)$-space. In particular, if $E$ is an infinite-dimensional metrizable non-normable lcs, then $E$ is a  quasi-$(DF)$-space which is not a $(DF)$-space.

(ii) Since $E$ is regular, the space $E'_\beta$ has a $\GG$-base by Proposition \ref{p:dual-regular-LM}. Therefore $E'_\beta$ is in the class $\GG$ (recall that every lcs $E$ with a $\GG$ base $\{ U_\alpha:\alpha\in\NN^\NN\}$ belongs to the class $\GG$ since the family of polars $\{ U_\alpha^\circ:\alpha\in\NN^\NN\}$ is a $\GG$-representation of $E$). As $E$ has a $\GG$-base and is quasibarrelled ($E$ is even bornological, see \cite[Corollary~13.1.5]{Jar}), the space $E'_\beta$ has a fundamental bounded resolution by Proposition \ref{p:quasibarrelled-strong-dual-res}. Therefore $E'_{\beta}$ is a quasi-$(DF)$-space.

(iii)-(v)  follow from  \cite[Proposition 11.1]{kak} and Proposition \ref{p:Stability-FBR}, and (vi) follows from \cite[Theorem 11.1]{kak}.
\end{proof}

Next two  examples show the independence of conditions (i) and (ii) appearing  in Definition \ref{def:quasi-DF} of quasi-$(DF)$-spaces.
\begin{example} \label{exa:def-DF-space-1} {\em
There exist lcs  $E$ not being in class $\GG$ but having a fundamental bounded resolution.  Indeed, if $E$ is an infinite-dimensional Banach space, then $E_{w}$ is not in class $\GG$ (see \cite{kak})  although $E_{w}$ has a fundamental sequence of bounded sets: Assume $E_{w}$ is in class $\GG$. Then, as $E_{w}$ is dense in $\mathbb{R}^{X}$ for some $X$, the Baire space $\mathbb{R}^{X}$  belongs also to the class $\GG$. Now the main theorem of \cite{KM} applies to deduce that $X$ is countable, a contradiction (since then $E_{w}$ would be metrizable implying  the finite-dimensionality of $E$).}
\end{example}

\begin{example} \label{exa:def-DF-space-2}
Let $X$ be a non $\sigma$-compact \v{C}ech-complete Lindel\"{o}f space (for example, $X=\NN^\NN$). Then $\CC(X)$ has a $\GG$-base (hence is in the class $\GG$) and is barrelled but it does not have even a bounded resolution. Consequently, $\CC(X)$ is not a quasi-$(DF)$-space and  the strong dual of $\CC(X)$ does not have a $\GG$-base.
\end{example}

\begin{proof}
Since $X$ has a  fundamental compact resolution by (see Fact 1 in the proof of Proposition 4.7 in \cite{GKKLP}), the space $\CC(X)$ has a $\GG$-base by \cite{feka}. As $X$ is a $\mu$-space, $\CC(X)$ is barrelled.  On the other hand, assume that $\CC(X)$  has a bounded resolution. Then $C_p(X)$ has a bounded resolution too.  Since $X$ is \v{C}ech-complete and Lindel\"{o}f,  there exists (well known fact) a perfect map $T$ from $X$ onto a Polish space $Y$.   As $X$ is not $\sigma$-compact, then $Y$ is also not $\sigma$-compact.  Since $T$ is onto, the adjoint map $T^\ast: C_p(Y)\to C_p(X)$, $T^\ast(f)=f\circ T$, of $T$ is an embedding. Therefore $C_p(Y)$ has a bounded resolution. But this is impossible,  since then  $Y$  would be $\sigma$-compact by Corollary 9.2 of \cite{kak}. The last assertion follows from Theorem \ref{t:G-base-strong}.
\end{proof}

In \cite{Mich} it is proved  that if $E$ is a Banach space whose strong dual is separable, then $E_{w}$ is an $\aleph_{0}$-space.  In \cite[Corollary 5.6]{GKKM} it was shown that a Banach space which does not contain an isomorphic copy of $\ell_{1}$ has separable dual if and only if $E_{w}$ is an $\aleph_{0}$-space. Next theorem extends this result to quasi-$(DF)$-spaces.

\begin{theorem}\label{t:quasi-DF-cosmic}
Let $E$ be a quasi-$(DF)$-space.
\begin{enumerate}
\item [(i)] If the  strong dual $E'_{\beta}$ is separable, then $E_{w}$ is cosmic. %Consequently, this holds, for example,  if $E$ is a regular $(LM)$-space.
\item[(ii)] If the strong dual $E'_{\beta}$ is separable and barrelled, then $E_{w}$ is an $\aleph_{0}$-space.
\item[(iii)] If $E$ is a strict $(LF)$-space  such that $E'_{\beta}$ is separable, then $E_{w}$ is an $\aleph_{0}$-space.
\end{enumerate}
\end{theorem}

\begin{proof}
(i) By Theorem \ref{t:G-base-strong}, the space $E'_{\beta}$ has a $\GG$-base $\{U_{\alpha}:\alpha\in\mathbb{N}^{\mathbb{N}}\}$. Hence its polar $\{ U^\circ_{\alpha}:\alpha\in\mathbb{N}^{\mathbb{N}}\}$ forms a compact resolution in $E''_w :=(E'',\sigma(E'',E'))$. Since $E'_{\beta}$ is separable, $\sigma(E'',E')$ admits a weaker metrizable topology, and hence the space $E''_w$ is analytic by  \cite[Theorem 15]{CasOr} (i.e. $E''_w$ is a continuous image on $\NN^\NN$). Therefore $E''_w$ is a cosmic space. As $E_w$ is a subspace of $E''_w$, the space $E_{w}$ is cosmic as well.
%The last part  follows from Proposition \ref{p:dual-regular-LM}, as every strict $(LF)$-space is regular, see \cite{bonet}.

 (ii) As in (i), the space $E''_w$  has a compact resolution $\{ U^\circ_{\alpha}:\alpha\in\mathbb{N}^{\mathbb{N}}\}$. We show that $E''_w$ has a fundamental compact resolution. Indeed, Theorem \ref{t:G-base-strong} and the fact that $E''_w$ is locally complete (since $E'_{\beta}$ is barrelled) imply that $E''_w$ has a fundamental bounded resolution. Now, by the barrelledness of $E'_{\beta}$, the space $E'_{\beta}$ has a $\GG$-base $\U=\left\{ U_{\alpha }:\alpha \in \mathbb{N}^{\mathbb{N}}\right\} $. Therefore the family $\U^\circ :=\left\{ U_{\alpha }^{\circ}:\alpha \in \mathbb{N}^{\mathbb{N}}\right\} $ is a resolution for $E''_w$ consisting of compact subsets. To check that $\U^\circ$ swallows the compact sets,  let $K$ be a compact subset of $E''_w$. As $E'_{\beta}$ is barrelled, $K$ is equicontinuous. So there is an $\alpha\in\NN^\NN$ such that $K\subseteq U_\alpha^\circ$, and hence $\U^\circ$ swallows the compact sets. On the other hand, $E''_w$ is submetrizable since $E'_{\beta}$ is separable. Now Theorem 3.6 of \cite{COT-1} yields that $E''_w $ is an $\aleph_{0}$-space, so $E_{w}$ is an $\aleph_{0}$-space, too.

%(ii) As in (i), the space $E''_w$  has a compact resolution $\{ U^\circ_{\alpha}:\alpha\in\mathbb{N}^{\mathbb{N}}\}$. Since $E'_{\beta}$ is  barrelled, by Theorem \ref{t:barrelled-G-base}, $E''_w$ has a fundamental compact resolution. On the other hand, $E''_w$ is submetrizable since $E'_{\beta}$ is separable. Now Theorem 3.6 of \cite{COT-1} imply that $E''_w $ is an $\aleph_{0}$-space. Hence the subspace $E_{w}$ of $E''_w$ also is an $\aleph_{0}$-space.

(iii) Any  strict $(LF)$-space $E$, being regular, is a quasi-$(DF)$-space by (i) of Theorem \ref{t:properties-quasi-DF-spaces}. Note also that any $(LM)$-space is quasibarrelled (even bornological).  Thus to apply (ii) it is sufficient to show that $E'_{\beta}$ is barrelled.
Let $\{ E_{n}\}_{n\in\NN}$ be a defining sequence of Fr\'echet lcs for $E$.   For each $n\in\NN$, the strong dual $(E_{n})'_{\beta}$  of $E_{n}$  is a $(DF)$-space. Since $E$ is a strict limit, the dual $E'_{\beta}$  is linearly homeomorphic with the projective limit of the sequence $\{ (E_{n})'_{\beta})\}_{n\in\NN}$  of complete $(DF)$-spaces, see \cite[Preliminaries]{bonet2}. Therefore $E'_{\beta}$ is also complete. On the other hand, $E'_{\beta}$ is  continuously mapped onto each $(E'_{n})_{\beta}$, so any $(E'_{n})_{\beta}$ is separable. By \cite[Proposition 8.3.45]{bonet}, any $E_{n}$ is distinguished. Hence applying again \cite[Preliminaries~(c)]{bonet2}, the space $E'_{\beta}$ is quasibarrelled. Since any complete quasibarrelled space is barrelled (see \cite[Proposition~11.2.4]{Jar}), the space $E'_{\beta}$ is barrelled.
\end{proof}

Below we provide more concrete examples which clarify the fundamental differences between $(DF)$-spaces and quasi-$(DF)$-spaces.
\begin{example}
The space of distributions $D'(\Omega)$ over an  open non-empty subset $\Omega$ of $\mathbb{R}^{n}$ has the following properties:
\begin{enumerate}
\item [(i)]  $D'(\Omega)$ is a quasi-$(DF)$-space.
\item [(ii)]  $D'(\Omega)$  is a weakly $\aleph_{0}$-space.
\item [(iii)]  $D'(\Omega)$ is not a $(DF)$-space.
\item [ (iv)]  $D'(\Omega)$  is not a weakly Ascoli space.
\end{enumerate}
\end{example}

\begin{proof}
Recall that the space $D'(\Omega)$ is the strong dual of the space  $D(\Omega)$ of test functions which is a complete Montel (hence barrelled)   strict $(LF)$-space of a sequence of Montel--Fr\'echet lcs. %, see \cite{horvath}.
Therefore $D'(\Omega)$ is a quasi-$(DF)$-space by (ii) of Theorem      \ref{t:properties-quasi-DF-spaces}. %Since $D(\Omega)$  has a $\GG$-base, Proposition \ref{p:quasibarrelled-strong-dual-res} implies that $D'(\Omega)$ has a fundamental bounded resolution. Taking into account that $D'(\Omega)$ has a $\GG$-base by Theorem 3 of \cite{GKL2}, we obtain that  $D'(\Omega)$ is a quasi-$(DF)$-space.
Since $D'(\Omega)$ is a Montel quasi-$(DF)$-space whose strong dual $D(\Omega)$  is separable and barrelled, $D'(\Omega)$  is a weakly $\aleph_{0}$-space by (ii) of Theorem \ref{t:quasi-DF-cosmic}. As $D'(\Omega)$ does not have a fundamental bounded sequence (otherwise  $D(\Omega)$ would be metrizable),  $D'(\Omega)$ is not a $(DF)$-space. Finally, Theorem 1.6 of \cite{Gabr-LCS-Ascoli} states that if $E$ is a barrelled weakly Ascoli space, then every weak${}^\ast$-bounded subset of $E'$ is finite-dimensional.  Thus $D'(\Omega)$  is not  weakly Ascoli since $D(\Omega)$  has an infinite-dimensional compact sets.
\end{proof}

\begin{example} \label{exa:DF-no-G-base}  {\em
(Under $\aleph_1 <\mathfrak{b}$) The space $\CC\big( \w_1 \big)$ is a $(DF)$-space by Theorem 12.6.4 of \cite{Jar}  which is not barrelled (recall that the ordinal space $w_1 =[0,\w_1)$ is pseudocompact). However,  $\CC\big( \w_1 \big)$ does not have a $\GG$-base by Proposition 16.13 of \cite{kak}.} %This example also shows  that the barrelledness condition in Corollary \ref{c:barrelled-strong-dual} is essential.}
\end{example}

As we mentioned above, it is known that the space $C_{p}(X)$ is a $(DF)$-space if and only if $X$ is finite. Indeed, although always $C_{p}(X)$ is quasibarrelled, see
\cite[Corollary 11.7.3]{Jar}, the space $C_{p}(X)$ admits a fundamental sequence of bounded sets only  if $X$ is finite, see \cite{kak}. For quasi-$(DF)$-spaces $C_{p}(X)$ the corresponding situation  looks even more striking as the following theorem shows.
\begin{theorem} \label{t:Cp-quasi-DF}
For $C_{p}(X)$ the following assertions are equivalent:
\begin{enumerate}
\item [(i)] $C_{p}(X)$ is a quasi-$(DF)$-space.
\item[(ii)] $C_{p}(X)$ has a fundamental bounded resolution.
\item [(iii)] $X$ is countable.
\item[(iv)] $C_{p}(X)$ is in the class $\GG$.
\item [(v)] $C_{p}(X)$ has  a $\GG$-base.
\end{enumerate}
\end{theorem}
\begin{proof}
(ii) follows from (i) by the definition. (iii) follows from (ii) by Theorem \ref{t:FBR-Cp}. The implication (iii) $\Rightarrow$ (iv) is trivial. Since $C_{p}(X)$ is always quasibarrelled (see again \cite[Corollary 11.7.3]{Jar}), the implication (iv) $\Rightarrow$ (v) follows from \cite{CKS}. Finally, if $C_{p}(X)$ has a $\GG$-base, it is metrizable again by \cite{CKS}.
\end{proof}

Theorem \ref{t:Cp-quasi-DF} suggests the following question: {\em For which Tychonoff space $X$, the space $\CC(X)$ is a quasi-$(DF)$-space}?
Below we obtain a complete answer to this question for metrizable spaces $X$.
\begin{proposition} \label{p:Ck-quasi-DF-metr}
Let $X$ be a metrizable space. Then $\CC(X)$ is a quasi-$(DF)$-space if and only if $X$ is a Polish $\sigma$-compact space. In particular, if $X$ is  a Polish $\sigma$-compact but non-compact space, then $\CC(X)$ is a  quasi-$(DF)$-space which is not a $(DF)$-space.
\end{proposition}

\begin{proof}
Assume that  $\CC(X)$ is a quasi-$(DF)$-space. Since $\CC(X)$ has a fundamental bounded resolution, Corollary \ref{c:Ck-strong-dual-G} implies that $X$ is a $\sigma$-compact space. On the other hand, as $\CC(X)$ is barrelled, $\CC(X)$ belongs to the class $\GG$ if and only if it has a $\GG$-base, see Lemma 15.2 of \cite{kak}. Therefore $X$ has a fundamental compact resolution by \cite{feka}. Thus $X$ is Polish by the Christensen theorem \cite[Theorem~6.1]{kak}.

Conversely, if $X$ is a Polish $\sigma$-compact space, then $\CC(X)$ has a $\GG$-base by \cite{feka} and has a fundamental bounded resolution by Corollary \ref{c:Ck-strong-dual-G}. Thus $\CC(X)$ is a quasi-$(DF)$-space.

If the Polish $\sigma$-compact $X$ is not compact, it has a countable non relatively compact subset. Thus $\CC(X)$ is a $(DF)$-space by Theorem 10.1.22 of \cite{bonet}.
\end{proof}
Recall that $C_{k}(X)$ is a $(DF)$-space if and only if any countable union of compact subsets of $X$ is a relatively compact set, see \cite[Theorem 10.1.22]{bonet}.
It is  known that a $(DF)$-space $E$ is quasibarrelled if and only if $E$ has countable tightness, see \cite[Proposition 16.4 and Theorem 12.3]{kak}. Therefore any vector subspace of a quasibarrelled $(DF)$-space has the same property. This may suggest also the following
\begin{question}\label{df}
Does there exist a  quasi-$(DF)$-space with countable tightness not being quasibarrelled?
\end{question}

\bibliographystyle{amsplain}

\end{document}